%% file: chain_rule.tex
\documentclass[11pt]{article}

\usepackage{fullpage}
\usepackage[T1]{fontenc}
\usepackage[utf8]{inputenc}
\usepackage{desmacro}
\usepackage{verbatim}
\input{despkg}

\newcommand{\gau}[1]{\mathfrak{g}(#1)}
\newcommand{\ber}[1]{\mathfrak{b}(#1)}
\newcommand{\comr}[1]{\mathfrak{r}^{\mathfrak{b}}(#1)}
\newcommand{\comg}[1]{\mathfrak{r}^{\mathfrak{g}}(#1)}

\newtheorem{definition}{Definition}
\newtheorem{theorem}{Theorem}
\newtheorem{lemma}{Lemma}
\newtheorem{proposition}{Proposition}

\newtheorem{remark}{Remark}

\begin{document}
	
	\title{A Chain Rule for the Expected Suprema of Bernoulli Processes\thanks{This research was suppoorted in part by DARPA under the Learning with Less Labels (LwLL) program.}}
	
	\author{Yifeng~Chu\thanks{University of Illinois, Urbana, IL, USA. E-mail: ychu26@illinois.edu.}
	  \and Maxim~Raginsky\thanks{University of Illinois, Urbana, IL, USA. E-mail: maxim@illinois.edu}}
	
	\date{}
	\maketitle

		\begin{abstract}
		
		We obtain an upper bound on the expected supremum of a Bernoulli process indexed by the image of an index set under a uniformly Lipschitz function class in terms of properties of the index set and the function class, extending an earlier result of Maurer for Gaussian processes. The proof makes essential use of recent results of Bednorz and Lata{\l}a on the boundedness of Bernoulli processes.
			\end{abstract}

\section{Introduction}

Sharp bounds on the suprema of Bernoulli processes are ubiquitous in applications of probability theory. For example, in statistical learning theory, they arise as so-called \textit{Rademacher complexities} that quantify the effective ``size'' of a function class in the context of a given learning task \cite{bartlett_rademacher_2002}. To evaluate the generalization
error bound of a learning algorithm, one usually needs to obtain a good estimate
on the (empirical) Rademacher complexity of a suitable function class. In many scenarios, the  function class at hand is composite, and separate assumptions are imposed on
the constituent classes forming the composition. Thus, it is desirable to obtain a bound that takes into account the properties of the function classes in the composition. 
 
Let $(\varepsilon_i)$ and $(\xi_i)$ be sequences of i.i.d.\ symmetric Bernoulli (or Rademacher) variables, \ie, $\prob{\varepsilon_i=\pm 1}=\frac{1}{2}$, and
a sequence of i.i.d.\ standard Gaussian variables, respectively. The Rademacher and the Gaussian complexities of
$T \subset \ell^2(I)$, with $I$ countable, are defined by 
\begin{align*}
 &\ber{T}:= \expect{\sup_{t \in T} \sum_{i \in I} \varepsilon_i t_i},
 &\gau{T}:= \expect{\sup_{t \in T} \sum_{i \in I} \xi_i t_i},
\end{align*}
(see, \eg, \cite{talagrand_upper_2014,bednorz_boundedness_2014}). For $T \subset \Real^{k\times n} $, 
we follow \cite{maurer_vector-contraction_2016} and let
\begin{align*}
 &\ber{T}:= \expect{\sup_{t \in T} \sum_{i=1}^{n} 
 \sum_{j=1}^{k}\varepsilon_{ji} t_{ji}},
 &\gau{T}:= \expect{\sup_{t \in T} \sum_{i=1}^{n} 
 \sum_{j=1}^{k} \xi_{ji} t_{ji}},
\end{align*}
where $\varepsilon_{ji}$ and $\xi_{ji}$ are independent for different values of the pair of indices $(j,i)$. Given a class $\eF$ of functions $f: \Real^k \to \Real^\ell$ 
and $T\subset \Real^{k \times n}$, 
we form the composite class $\eF(T) := \set{f(t): f\in \eF, t \in T} \subset \Real^{\ell \times n}$ where $f(t) := (f(t_i))_{1\le i\le n}$. (In particular, if the functions in $\eF$ are real-valued, then $\eF(T)$ is a subset of $\Real^n$.) We would like to obtain an upper bound on the Rademacher complexity $\ber{\eF(T)}$ of the composite class $\eF(T)$ in terms of Rademacher complexities of $\eF$ and $T$ alone, under appropriate regularity conditions on $\eF$ and $T$.  A chain rule of this type was obtained by Maurer \cite{maurer_chain_2014} for the Gaussian complexity $\gau{\eF(T)}$. In this work, we are concerned with chain rules for Rademacher complexities. This extension relies on completely different techniques, in particular on recent powerful results of Bednorz and Lata{\l}a \cite{bednorz_boundedness_2014} on the so-called \textit{Bernoulli conjecture} of Talagrand.

We first introduce some notation. For $p, q \in [1,\infty]$ and $t=(t_1,\ldots,t_n) \in \Real^{k \times n}$ with each $t_i \in \Real^k$, 
let $\norm{t}_{p,q} := \big\|\big(\norm{t_1}_p,\ldots,\norm{t_n}_p\big)\big\|_{q}$, i.e., $\norm{t}_{p,q}$ is the $q$-norm of the vector whose entries are the $p$-norms of the columns of $t$. We let $\norm{t}_2 = \norm{t}_{2,2} =\norm{t}_{F}$, where $\norm{\cdot}_{F}$ denotes the
Frobenius norm. For a set $T \subset \Real^{k \times n}$, we let $\norm{T}_{p,q} := \sup_{t \in T} \norm{t}_{p,q}$. 
The Lipschitz constant of a function $f:\Real^k \to \Real$ w.r.t.\ the Euclidean norm will be denoted by $\norm{f}_{\Lip}$. We will write $a \lesssim b$ if $a \le Cb$ for a universal constant $C$; the cases when the constant $C$ depends on some parameter will be indicated explicitly. Our main result is as follows:

\begin{theorem} \label{main_res_a}
  Let a countable\footnote{We restrict attention to countable classes in order to avoid dealing with various measure-theoretic technicalities, which can be handled in a standard manner by, e.g., imposing separability assumptions on the relevant function classes and processes.} class $\eF$ of functions $f : \Real^k \to \Real$ and a countable
  set $T \subset \Real^{k \times n}$ be given.
  Assume that all $f \in \eF$ are uniformly Lipschitz with respect to the Euclidean norm,
  i.e., $\sup_{f \in \eF} \norm{f}_{\Lip} \le L < \infty$. Assume also that  $\ber{T}<\infty$ and $ \ber{\eF(T)}< \infty$.
  Then there exists a set $S \subset \Real^{k \times n}$ with  $\gau{S}\le \ber{T}$
  and $\norm{S}_{\infty,\infty} \le \norm{T}_{\infty,\infty}$, such that
      \begin{align}\label{eq:ber_chain_rule}
        \ber{\eF(T)}  \lesssim L\ber{T}  + \Delta_2(S)\comr{\eF, S} +  \ber{\eF(\tau)}
      \end{align} 
      for any $\tau \in S$, where $\Delta_2(S)$ is the diameter of $S$ with respect to the Frobenius norm, and where
      \begin{align*}
        \comr{\eF, S}:= \sup_{s, t \in S}
        \frac{\expect{\sup_{f \in \eF}
        \sum_{i=1}^{n} \varepsilon_i(f(s_i)-f(t_i))}}{\norm{s-t}_2} .
      \end{align*}
\end{theorem}

  Theorem~\ref{main_res_a} is a counterpart of Theorem 2 in \cite{maurer_chain_2014} on
  expected suprema of Gaussian processes. There are, however, two differences between the results of Maurer and ours. One is that Maurer considers the case when $T$ is a subset of $\Real^\ell$ and $\eF$ is a uniformly Lipschitz class of functions from $\Real^\ell$ to $\Real^n$. Our result carries over to this setting with a straightforward modification, with the only difference that the functions in $\eF$ have to be uniformly Lipschitz in both $\ell^1$ and $\ell^2$ norms. This extra requirement is expected since the suprema of Bernoulli
  processes can be controlled in two ways that are of different nature, as we will discuss
  in the next section. Another important difference is that the chain rule for the suprema of Gaussian processes has the form
      \begin{align}\label{eq:gau_chain_rule}
        \gau{\eF(T)} \lesssim L\gau{T}  + \Delta_2(T)\comg{\eF, T} + \gau{\eF(\tau)} ,
      \end{align}
where $\comg{\eF,\cdot}$ is defined analogously to $\comr{\eF,\cdot}$, but with $(\xi_i)$ instead of $(\varepsilon_i)$ \cite{maurer_chain_2014}. Notice that \eqref{eq:ber_chain_rule} involves a set $S$ which is different from $T$, whereas $T$ appears on both sides of \eqref{eq:gau_chain_rule}. However, using the inequality $\Delta_2(S) \le \sqrt{2\pi}\gau{S}$ \cite[p.~44]{talagrand_upper_2014} and  the fact that $\gau{S} \le \ber{T}$, we can obtain the following slightly weaker form of \eqref{eq:ber_chain_rule}:
      \begin{align*}
        \ber{\eF(T)}  \lesssim (L + \comr{\eF, S})\ber{T} +  \ber{\eF(\tau)}.
      \end{align*} 

Although concrete bounds on $\comg{\eF,S}$ or \(\comr{\eF,S}\)
can be obtained in particular cases
as in \cite{maurer_chain_2014},
it is nontrivial to obtain good general estimates for these quantities.
This motivates the next result, which involves \textit{entropy numbers} of $\eF$ \cite{talagrand_upper_2014}. For any set $\sX \subset \Real^k$,
let $\norm{f}_{\sX} := \sup_{x \in \sX}|f(x)|$ be the uniform norm with respect to $\sX$, and let $d_{\sX}$ be the induced metric.
 For any bounded metric space $(T,d)$, let $\Delta(T,d)$ be the diameter of $T$. For $S \subset T$ and $t \in T$, let $d(t, S) := \inf_{s\in S} d(t,s)$. For any integer $m\ge 0$, define the $m$th \textit{entropy number}\footnote{Note that in, some texts, entropy numbers
refer to logarithms of covering numbers. Some relations between these two definitions can be found
in \cite{talagrand_upper_2014}.} of $T$.
\begin{align*}
  e_m(T,d):= \inf \sup_{t \in T} d(t, T_m),
\end{align*}
where the infimum is taken over all subsets $T_m$ of $T$ such that $\abs{T_0}= 1$ and 
$\abs{T_m} \le 2^{2^m}$. 

\begin{theorem}\label{main_res_b} Under the same assumptions as in Theorem~\ref{main_res_a}, let $\sS=\bigcup_{i=1}^{n} S_i $ where $S_i \subset \Real^k$
  is the projection of $S$ onto the $i$th coordinate.  Then
\begin{align*}
  \ber{\eF(T)} \lesssim L \ber{T} + n \inf_{M \in \Natural} 
  \set{e_M(\eF,d_{\sS}) +  \sum_{m=0}^{M} \frac{2^{m / 2}}{\sqrt{n}}  e_m(\eF,d_{\sS})}
\end{align*}
In particular, if we take $\sX=\eB_{2}^{k}(R) \supset \sS$, where $\eB_{2}^{k}(R) \subset \Real^k$ is a closed $\ell^2$ ball of radius $R$
centered at $0$, then for any uniformly Lipschitz class $\eF$ with $\norm{f}_{\sX}\le LR$,
we have
\begin{align} \label{crude_est}
  \ber{\eF(T)} \le cL \ber{T} + C_k L R r_{n,k} \mbox{ with }
  \frac{r_{n,k}}{n} =  
\begin{cases}
 n^{- 1 / 2}, & \text{ for } k=1\\
 n^{-1 / 2} \log n,& \text{ for } k=2\\
 n^{-1 / k},& \text{ for } k>2
\end{cases},
\end{align}
where $c$ is a universal constant and $C_k$ depends on $k$.
\end{theorem}

\begin{remark}
  Note that the radius $R$ may depend on $k$, since we need to make sure that $\sX$ is large enough to include $\sS$.
  A possible choice of $R$ is given as follows.
  Since we assume $\ber{T}< \infty$, we can find an $R' < \infty$ such that
  $\Delta_2(T)\le R'$. Then we can take $R=\sqrt{k} R'$ due to 
  the fact that $\norm{S}_{2,\infty} \le \sqrt{k} \norm{S}_{\infty,\infty} 
 \le \sqrt{k} \norm{T}_{\infty,\infty} \le \sqrt{k} \norm{T}_{2,\infty}$. This will introduce an extra $k$-dependent factor for general $\eF$ as in (\ref{crude_est}), but it is not necessarily the case if we have
 finer estimates on the entropy numbers. It is also worth mentioning
 that if $\eB_2^{k}(R')$ contains a $k$-dimensional cube which contains $T$, we can simply take
 $R=R'$.
\end{remark}

The remainder of the paper is organized as follows: \Cref{sec:prelim} introduces some preliminaries on chaining and
its application to bounding expected suprema of Gaussian and Bernoulli processes; \Cref{sec:proofs}
gives detailed proofs of our main results; \Cref{sec:examples} gives some examples on how the main results
can be applied; proofs of some intermediate technical results are collected in \Cref{app:aux}.

\section{Preliminaries}
\label{sec:prelim}

Let $I \subset \Natural$ and $T \subset \ell^2(I)$. 
We first gather some useful facts on $\ber{T}$ and $\gau{T}$:
\begin{lemma}[]
  For any $T \subset \ell^2(I)$, we have $\ber{T} \le \sup_{t \in T} \norm{t}_1$.
\end{lemma}
\begin{lemma}[\cite{bednorz_boundedness_2014}]
  For any $T \subset \ell^2(I)$, we have $\Delta_2(T) \le 4\ber{T}$ and
  $\ber{T} \le \sqrt{\frac{\pi}{2}} \gau{T}$.
\end{lemma}
\begin{lemma}[Eq.~(4.9) in \cite{ledoux_probability_1991}]\label{ber_dom_gau}
  For any $T \subset \Real^m$, we have $\gau{T} \lesssim \sqrt{\log m} \cdot \ber{T}.$
\end{lemma}
\noindent In general, when we use $\ber{T}$ to control $\gau{T}$, the inequality of Lemma~\ref{ber_dom_gau} cannot
be improved. Obviously, we can use Lemma~\ref{ber_dom_gau} with the chain rule 
for $\gau{\eF(T)}$ from \cite{maurer_chain_2014} to obtain an upper bound on $\ber{\eF(T)}$, but that
would introduce an extra logarithmic factor. Thus we need to modify the techniques from
\cite{maurer_chain_2014} to prove our results. Before that, let us introduce
the most important device used in \cite{maurer_chain_2014} and the current work: generic chaining, which
is ubiquitous in the study of suprema of stochastic processes \cite{talagrand_upper_2014}. We start by giving the requisite definitions. 
\begin{definition}[]
  Given a sequence of sets $(B_m)_{m\ge 0}$, if $\abs{B_0}=1$ and $\abs{{B_m}}\le 2^{2^m}$
  for $m\ge 1$,
  we say each $B_m$ has \textit{admissible} cardinality.
\end{definition}
\begin{definition}[]
Given a set $T$, a sequence of partitions $(\cA_m)$ of $T$ is \textit{increasing} if, for any
$m\ge 0$, every set of $\cA_{m+1}$ is contained in a set of $\cA_m$. Then
an \textit{admissible sequence} of $T$ is an increasing sequence $(\cA_m)$ of 
partitions of $T$ such that each $\cA_m$ has admissible cardinality. 
\end{definition}
 \begin{definition}[]
  Given a metric space $(T,d)$, let  
\begin{align*}
  \gamma_2(T,d) :=\inf \sup_{t \in T} \sum_{m\ge 0} 2^{m / 2} \Delta(A_m(t),d),
\end{align*}
where the infimum is taken over all admissible sequences of $(T,d)$, and
$A_m(t)$ is the unique subset of $\cA_m$ containing $t$. 
\end{definition}

Now we state a bound on expected suprema of stochastic processes that satisfy a certain increment
condition. Note that it can be recognized as a minor modification of the \textit{generic chaining bound} (see \eg \cite{talagrand_upper_2014}). A similar statement can also be found in \cite[Theorem 3]{maurer_chain_2014}.
\begin{lemma}\label{gen_chain_appr} Let $(T,d)$ be a metric space. Let $(Y_t)_{t \in T}$ be a stochastic process indexed by $T$ that
  satisfies the increment condition: there exists some $\kappa \ge 1$ such that, for all $s,t \in T$ and all $u > 0$,
\begin{align}\label{subgau_cond}
  \prob{Y_s-Y_t> u} \le \kappa \exp\left(-\frac{u^2}{2d(s,t)^2}\right).
\end{align}
Then, for any $\tau \in T$ and for each admissible sequence $(\cA_m)$ of $T$,
\begin{align}
  \expect{\sup_{t \in T} Y_t - Y_{\tau}} \lesssim  
  \sup_{t \in T} \sum_{m=0}^{\infty} 2 ^{m / 2} \Delta(A_m(t),d) + 
 \Delta(T,d) \sqrt{\log \kappa} .
\end{align}
\end{lemma}
\noindent In particular, if we take $T=\ell^2(I)$, $d=d_2$, and $Y \colon t \mapsto \sum_{i \in I} \xi_i t_i$,
we can apply the result above with $\kappa=1$ and take the infimum on both sides over all
 admissible sequences to obtain
\begin{align}\label{gen_chain_gau}
  \gau{T}\lesssim \gamma_2(T,d_2).
\end{align}
It turns out that (\ref{gen_chain_gau}) can be reversed by the \textit{majorizing measure theorem} \cite{talagrand_upper_2014}:
\begin{theorem}\label{major_meas}
  There exists an admissible sequence $\cA_m$ of $T$ such that
\begin{align*}  
  \sup_{t \in T} \sum_{m\ge 0} 2^{m / 2} \Delta_2(A_m(t)) \lesssim \gau{T} .
\end{align*}
\end{theorem}

Turning now to Bernoulli averages, let us define
\begin{align*}
  \mathfrak{b}^*(T) := \inf \set{\sup_{t \in T_1} \norm{t}_1 + \gamma_2(T_2,d_2): T \subset T_1 + T_2  },
\end{align*}
where $T_1+ T_2$ is the Minkowski sum, \ie, $T_1 + T_2 = \set{ t_1+ t_2: t_1\in T_1, t_2\in T_2 }$. 
 Since $\ber{T} \le \ber{T_1+T_2}\le \ber{T_1}+\ber{T_2}\le\sup_{t \in T_1} \norm{t_1} + c\gamma_2(T_2,d_2)  $ for a universal constant $c$,
 we obtain $\ber{T} \lesssim \mathfrak{b}^*(T)$.
Whether such an estimate can be reversed, known as Talagrand's Bernoulli Conjecture,
remained an open problem for over two decades and has been
affirmatively answered by the following theorem 
(see \eg \cite[Theorem 1.1]{bednorz_boundedness_2014}).
\begin{theorem}[Bednorz-Lata\l{}a]\label{bed_lata}
For $\ber{T} < \infty$, there exists a decomposition $T \subset T_1 + T_2$ with
\begin{align*}
  \sup_{t \in T_1} \sum_{i \in I} \abs{t_i} \lesssim \ber{T} \text{ and } \gau{T_2} \lesssim \ber{T}.
\end{align*}
\end{theorem}
\begin{remark}
The main difficulty when working with Bernoulli processes
stems from the fact that $\ber{T}$ can be controlled both by $\sup_{t \in T}\norm{t}_1$ and $\gau{T}$, which
are of completely distinct nature, and there is no  canonical decomposition of $T$ into a Minkowski sum $T_1 + T_2$. An
explicit construction of $T_1,T_2$ based on admissble sequences tailored to $\ber{T}$ can be
found in \cite[Theorem 3.1]{bednorz_boundedness_2014}. As we show  next, 
Theorem~\ref{bed_lata} is the
key ingredient needed to extend Maurer's chain rule to Bernoulli case.
\end{remark}

\section{Proofs of Main Results} 
\label{sec:proofs}

\sloppypar The main idea 
behind the proof of \cite[Theorem 2]{maurer_chain_2014} is to consider $t \mapsto \sup_{f \in \eF} \sum_{i \in I} \xi_i f(t_i)$ as
a stochastic process indexed by $T$. Upon verifying that this process satisfies a subgaussian
increment condition \eqref{subgau_cond}, one can apply the chaining argument. However, since chaining for
Gaussian processes does not
suffice to `explain' Bernoulli processes, we
first invoke the decomposition in the Bednorz--Lata\l{}a theorem and then
apply chaining only to the parts that are better explained by it.  With the decomposition in \Cref{bed_lata}, we have
the following crucial observation.
\begin{proposition}[]\label{decomp_comp}
 For $T\subset \Real^{k \times n}$, there exist $T_1, T_2 \subset \Real^{k\times n}$ 
 with $T \subset$ $T_1 +T_2$, such that
 \begin{align*}
   \ber{\eF(T)} \lesssim L \sup_{t \in T_1}\norm{t}_{1,1} 
   +  \ber{\eF(T_2)}\lesssim L\ber{T} +  \ber{\eF(T_2)},
 \end{align*}
 where $T_2$ satisfies $\gau{T_2}\lesssim \ber{T}$.
\end{proposition}
  
\begin{proof}
  For $T \subset \Real^{k\times n}$, we identify $T$ as $\overline{T} \subset \Real^{kn}$ by vectorization, \ie,
  for any $t \in T$, let
  \begin{align*}
    \mathsf{vec}: t =(t_1, \ldots, t_n)= \left( 
               \begin{pmatrix}
                t_{11}\\
                t_{12}\\
                \vdots\\
                t_{1k}
               \end{pmatrix}, \ldots, 
               \begin{pmatrix}
                t_{n1}\\
                t_{n2}\\
                \vdots\\
                t_{nk}
               \end{pmatrix}\right) \mapsto ((t_{11},\ldots,t_{1k}),\ldots, (t_{n1},\ldots, t_{nk})) \in \Real^{kn},
  \end{align*}
  and let $\mathsf{unvec}$ be its inverse. Thus, $\overline{T} = \mathsf{vec}(T)$, and we will use  $\overline{t}$ for elements of $\overline{T}$.  We construct $\overline{T}_1, \overline{T}_2$ as in \cite[Theorem 3.1]{bednorz_boundedness_2014} so that 
  \Cref{bed_lata} applies. In particular, as detailed in the proof of Theorem 1.1 of \cite{bednorz_boundedness_2014}, there exists a map $\pi : \Real^{kn} \to \Real^{kn}$, such that
 \begin{align*}
    \overline{T}_1 = \set{\overline{t}- \pi(\overline{t}): \overline{t}\in \overline{T}} 
    \quad\mbox{  and  }\quad \overline{T}_2 = \set{\pi(\overline{t}): \overline{t} \in \overline{T}},
  \end{align*}
  which are identified as $T_1, T_2$ in $\Real^{k \times n}$ by letting $T_1= \mathsf{unvec}(\overline{T}_1),T_2= \mathsf{unvec}(\overline{T}_2) $. 
With a slight abuse of notation, we also write $\pi(t)$ for $t \in T$
   by letting $\pi(t)=\mathsf{unvec}(\pi(\mathsf{vec}(t)))$.
  Then, we obtain
  \begin{align*}
    \ber{\eF(T)}&= 
     \expect{\sup_{f\in \eF, t \in T} 
    \sum_{i =1}^{n} \varepsilon_i f(t_i)}\\
    &=\expect{\sup_{f\in \eF, t \in T } \sum_{i=1}^{n} \varepsilon_i 
    \Big( f(t_i)-f(\pi(t)_i)+ f(\pi(t)_i) \Big)}\\
    &\le \expect{\sup_{f\in \eF,t \in T} \sum_{i =1}^{n} \varepsilon_i \Big(f(t_i)- f(\pi(t)_i)\Big)}
    + \expect{\sup_{f\in \eF,t \in T} \sum_{i =1}^{n} \varepsilon_i f(\pi(t)_i)}  .
  \end{align*}
  We recognize the second term as $\ber{\eF(T_2)}$ by the choice of $T_2$.
  For the first term,
  \begin{align*}
     \expect{\sup_{f\in \eF,t \in T} \sum_{i =1}^{n} \varepsilon_i \Big(f(t_i)- f(\pi(t)_i)\Big)}
        &\le \sup_{f\in \eF,t \in T} \sum_{i =1}^{n} \abs{f(t_i)- f(\pi(t)_i)} \\
        &\le L \sup_{t\in T_1} \sum_{i =1}^{n} \norm{t_i}_2\\
        &\le L \sup_{t\in T_1} \sum_{i =1}^{n} \norm{t_i}_1\\
        &=L \sup_{\overline{t} \in \overline{T}_1} \norm{\overline{t}}_1 = L \sup_{t \in T_1} \norm{t}_{1,1}
  \end{align*}
  where the inequalities are due to the choice of $T_1$, the Lipschitz continuity of $f$, and the relation between
  $\ell^2$ and $\ell^{1}$.
  Since $\overline{T}_1$ satisfies
  $\sup_{\overline{t} \in \overline{T}_1} \norm{\overline{t}}_1 \lesssim \ber{T}$ by \Cref{bed_lata}, the proof is complete.
\end{proof}

 \subsection{Proof of Theorem 1} 

The overall strategy of the proof now closely follows that of \cite[Theorem 2]{maurer_chain_2014}, although some care must be taken to control $\ber\cdot$ in terms of $\gau\cdot$. We first state a simple lemma \cite{maurer_chain_2014}:
 \begin{lemma}\label{uncenter}
   Suppose a random variable $Y$ satisfies $\prob{ Y- a > u } \le \e^{-u^2}$ for any $u>0$, where $a \in \Real$ is a fixed constant. Then $\prob{Y>u}\le \e^{a^2} \e^{-u^2 / 2}$ for all $u > 0$.
 \end{lemma}

\begin{proof}[\textit{Proof of \Cref{main_res_a}}] Let $T \subset T_1 + T_2$ be the decomposition constructed in Proposition~\ref{decomp_comp}. Define the stochastic process $Y$ indexed by $T_2$,
  \begin{align*}
    Y \colon t  \mapsto \frac{1}{\sqrt{2} L} \sup_{f \in \eF} 
    \sum_{i=1}^{n} \varepsilon_i f(t_i).
  \end{align*}
  Then $\ber{\eF(T_2)} = \sqrt{2} L\, \expect{ \sup_{t \in T_2} Y_t} $. For $s,t \in T_2$,
  \begin{align*}
  Y_{s} - Y_{t} \le \frac{1}{\sqrt{2} L} \sup_{f \in \eF}
  \sum_{i=1}^{n} \varepsilon_i (f(s_i)- f(t_i)) =: Z_{s,t}.
  \end{align*}
Using McDiarmid's inequality and Lemma~\ref{uncenter}, we obtain
\begin{align*}
  \prob{Y_{s} - Y_{t}>u }\le \prob{Z_{s,t} >u} \le \exp\left(\frac{(\mathbf{E}Z_{s,t})^2}{\norm{s-t}_2^2}\right)
  \exp\left(-\frac{u^2}{2\norm{s-t}^2_2}\right).
\end{align*}
 Thus, $(Y_t)_{t \in T_2}$ satisfies the increment condition of Lemma~\ref{gen_chain_appr} with $\kappa := \exp\left({\sup_{s,t \in T_2}} \frac{(\mathbf{E}Z_{s,t})^2}{\norm{s-t}_2^2}\right)$. Fix an arbitrary $\tau \in T_2$. Since the Frobenius norm for $T_2$ induces the canonical distance function for the Gaussian process
 $T_2\ni t\mapsto \sum_{i,j}\xi_{ij}t_{ij}$, we can use the majorizing measure theorem (\Cref{major_meas}) in conjunction with Lemma~\ref{gen_chain_appr} to obtain
\begin{align*}
  \expect{\sup_{t \in T_2} Y_t -Y_\tau}
 &\lesssim \gau{T_2} + \Delta_2(T_2)\sqrt{\log \kappa}.
\end{align*}
Substituting the definitions of $Y_t$ and $\kappa$ and rearranging, we get
\begin{align*}
  \ber{\eF(T_2)} &\lesssim L \gau{T_2} +  \Delta_2(T_2) 
  \comr{\eF,T_2}+ \ber{\eF(\tau)} \\
    &\lesssim L \ber{T} + \Delta_2(T_2)\comr{\eF,T_2}+ \ber{\eF(\tau)},
\end{align*}
where the last inequality follows from the fact that $\gau{T_2} \lesssim \ber{T}$ by 
\Cref{bed_lata}.
Combining the above result with Proposition~\ref{decomp_comp} concludes the proof of
\Cref{main_res_a} with $S = T_2$.
\end{proof}

\subsection{Proof of \Cref{main_res_b}}
Recall that, from Proposition~\ref{decomp_comp}, we obtain the bound $\ber{\eF(T)} \lesssim L \ber{T} + \ber{\eF(T_2)}$, where $T_2$ satisfies $\gau{T_2} \lesssim \ber{T}$.
In this subsection, we will use both classical (Kolmogorov) chaining and generic chaining
to estimate $\ber{\eF(T_2)}$. 

We first state an entropy number estimate for uniformly Lipschitz function classes
on a totally bounded domain (\cf \cite[Lemma 4.3.9]{talagrand_upper_2014}). 
\begin{lemma}\label{ent_num_est}
Let $(T,d)$ be metric space, such that, for a certain $B$ and $k\ge 1$, we have
    $N(T,d,\delta) < \left(\frac{B}{\delta}\right)^k$ for all $\delta > 0$, where $N(T,d,\delta)$ is the $\delta$-covering number for $(T,d)$, \ie, the minimum number of
  closed $d$-balls of radius $\delta$ needed to cover $T$. Consider the set 
  $\eF$ of $L$-Lipschitz functions $f : T \to \Real$ with $\norm{f}_{T}\le LB$. Then, for each integer $m \ge 0$,
  we have $e_m(\eF, d_{T})  \le C_k LB 2^{-m / k}$.
\end{lemma}
  Let $S=T_2$. 
 For $m\ge 0$, we choose $\eF_m \subset \eF$
 with admissible cardinality such that
 $\sup_{f \in \eF} d_{\sS}(f, \eF_m) \le 2  e_m(\eF, d_{\sS})$ and choose
 $S_m \subset T_2$ such that each element of $S_m$ is an arbitrary point in each of the elements of the partition $\cA_m$, where $(\cA_m)_{ m \ge 0}$ is an admissible sequence in $T_2$ satisfying 
$$
\sup_{t \in T_2} \sum_{m\ge 0} 2^{m / 2} \Delta_2(A_m(t)) \lesssim \gau{T_2}
$$
by the majorizing measure theorem.  Let $p_m: \eF \to \eF_m$ be such that $d_{\sS}(f, p_m(f))= d_{\sS}(f, \eF_m)$ for any $f\in \eF$, and let
 $\pi_m: T_2\to S_m$ be such that $d_2(t, \pi_m(t))= d_2(t, S_m)$ for any $t\in T_2$.
 Let $f_0 = p_0(f), \tau= \pi_0(t)$ for any $f\in \eF,t \in T_2$. 
 For an arbitrary $M \in \Natural$,  consider the quantities
 \begin{align*}
   &\mathtt{U_1}:= \expect{\sup_{f \in \eF, t\in T_2}  \sum_{i=1}^{n} 
   \varepsilon_i(f(t_i)- p_M(f)(t_i))}\\
   &\mathtt{U_2}:= \expect{\sup_{f \in \eF, t\in T_2}  \sum_{i=1}^{n} 
   \varepsilon_i(p_M(f)(t_i)- p_M(f)(\pi_M(t)_i))}\\
   &\mathtt{U_3}:= \expect{\sup_{f \in \eF, t\in T_2}  \sum_{i=1}^{n} 
   \varepsilon_i(p_M(f)(\pi_M(t)_i)- p_0(f)(\pi_0(t)_i))}.
 \end{align*}
It is obvious that 
 \begin{align*}
   \ber{\eF(T_2)}= \expect{\sup_{f,t} \sum_{i=1}^{n} \varepsilon_i (f(t_i)-f_0(\tau_i))}
   \le \mathtt{U_1}+\mathtt{U_2}+\mathtt{U_3}.
 \end{align*}
 For $\mathtt{U_1}$, by the Cauchy--Schwarz inequality
 \begin{align*}
   \mathtt{U_1} &\le \expect{ \left(\sum_{i=1}^{n} \varepsilon_i^2\right)^{1 / 2}
   \sup_{f,t}\left(\sum_{i=1}^{n} (f(t_i)-p_M(f)(t_i))^2\right)^{1 / 2}} \\
                &\le n \sup_{f \in \eF} d_{\sS}(f,p_M (f))\le cne_M(\eF,d_{\sS}).
 \end{align*}
Next, observe that
$$\mathtt{U_2}\le \expect{\sup_{f \in \eF_M, t\in T_2}
 \sum_{i=1}^{n} \varepsilon_i (f(t_i)- f(\pi_M(t)_i))}
 $$
 and consider the following stochastic process $Y$ indexed by $\eF_M \times T_2$: $Y: (f,t)\mapsto \sum_{i=1}^{n} \varepsilon_i f(t_i)$.  We will use generic chaining to bound the supremum $\expect{\sup_{f\in \eF_M, t\in T_2}(Y_{f,t}- Y_{f,\pi_M(t)})} $.  Let
 $$
 \varrho = L \sup_{t \in T_2}
 \sum_{m\ge M+1} 2^{m / 2}\norm{\pi_m(t)-\pi_{m-1}(t)}_2.
 $$
 By the union bound,
 \begin{align*}
   \prob{\sup_{f \in \eF_M, t\in T_2}(Y_{f,t}- Y_{f,\pi_M(t)})>u\varrho}
   \le \sum_{f \in \eF_M} \prob{\sup_{t \in T_2} (Y_{f,t}- Y_{f,\pi_M(t)})> u\varrho}, \qquad u > 0.
 \end{align*}
 Notice that, for any $f\in \eF_M$, the process $t\mapsto \sum_{i=1}^{n}\varepsilon_i f(t_i)$
 satisfies the increment condition
\begin{align*}
  \prob{Y_{f,s}-Y_{f,t} > u} \le \exp \left(-\frac{u^2}{2L^2\norm{s-t}_2^2}\right), \qquad  u > 0
\end{align*} 
for all $s,t \in T_2$. Using chaining and the union bound again, we obtain 
\begin{align*}
  \prob{\sup_{t \in T_2} (Y_{f,t}-Y_{f,\pi_M(t)})> u\varrho}
  &\le\prob{\sup_{t\in T_2} \sum_{m\ge M+1} (Y_{f,\pi_m(t)}- Y_{f,\pi_{m-1}(t)})> u\varrho}\\
  &\le \sum_{m\ge M+1} \sum_{s \in S_m} \prob{Y_{f,s}- Y_{f,\pi_{m-1}(s)} 
  > uL 2^{m / 2} \norm{s-\pi_{m-1}(s)}_2}\\
  &\le \sum_{m \ge M+1} 2^{2^{m+1}} \e^{-u^2 2^{m-1}}
\end{align*}
Therefore,
\begin{align*}
  \prob{\sup_{f \in \eF_M, t \in T_2}(Y_{f,t}-Y_{f,\pi_M(t)}) >u\varrho} 
  \le 2^{2^{M}}\left(\sum_{m \ge M+1} 2^{2^{m+1}} \e^{-u^2 2^{m-1}}\right)
  \le \sum_{m \ge M+1}2^{2^{m+2}} \e^{-u^2 2^{m-1}}.
\end{align*}
By Proposition~\ref{union_comp} and the choice of $(\pi_m(t))_{m\ge M}$, we then have
\begin{align*}
  \mathtt{U_2} \lesssim L \sup_{t \in T_2} \sum_{m\ge M} 2^{m / 2} \Delta_2(A_m(t)) \lesssim L \gau{T_2} \lesssim L\ber{T}.
\end{align*}
For $\mathtt{U_3}$, we still consider the process $Y: (f,t)\mapsto \sum_{i=1}^{n} f(t_i)$ indexed by \(\eF \times T_2\), but we now endow \(\eF \times T_2\) with the following metric:
for $f_1, f_2\in \eF, s, t\in T_2$,
\begin{align*}
  d_3((f_1,s),(f_2,t)) := \sqrt{n} d_{\sS}(f_1,f_2) + L\norm{s-t}_2.
\end{align*}
Since $d_3((f_1,s),(f_2,t))\ge \norm{f_1(s)-f_2(t)}_2$, this process satisfies the increment condition
\begin{align*}
  \prob{Y_{f_1,s}- Y_{f_2,t}> u} \le \exp\left(-\frac{u^2}{2d_3((f_1,s),(f_2,t))^2}\right), \qquad u > 0.
\end{align*}
By a generic chaining argument, we have
\begin{align*}
  \mathtt{U_3} &\lesssim \sup_{f\in \eF, t\in T_2} \sum_{m=1}^{M} 2^{m  / 2} 
  d_3((p_m(f),(\pi_m(t)),(\pi_{m-1}(f),\pi_{m-1}(t)))) \\
               &\lesssim \sqrt{n} \sup_{f \in \eF} \sum_{m=1}^{M} 2^{m / 2}d_{\sS}(p_m(f),p_{m-1}(f)) +
               L \sup_{t \in T_2} \sum_{m=1}^{M} 2^{m / 2} \norm{\pi_m(t)-\pi_{m-1}(t)}_2\\
               &\lesssim \sqrt{n} \sup_{f \in \eF} \sum_{m=1}^{M} 2^{m / 2}\big(d_{\sS}(f,p_m(f))
               +d_{\sS}(f,p_{m-1}(f))\big)
               + L\gau{T_2} \\
               &\lesssim \sqrt{n}  \sum_{m=0}^{M} 2^{m / 2} \sup_{f \in \eF}d_{\sS}(f, \eF_m) + L\ber{T}\\
               &\lesssim \sqrt{n} \sum_{m=1}^{M} e_m(\eF,d_{\sS}) + L\ber{T}.
\end{align*}
Combining the above results and taking the infimum over $M$, we obtain
\begin{align*}
  \ber{\eF(T_2)}\lesssim L \ber{T} + n\inf_{M}\set{e_M(\eF, d_{\sS}) + \sum_{m=0}^{M} \frac{2^{m / 2}}{\sqrt{n}}e_m(\eF,d_{\sS}) } .
\end{align*}

In particular, if $\sX= \eB_{2}^{k}(R) \supset \sS$, we have $N(\mathsf{X},d_2,\delta)\lesssim (R / \delta)^k$ by the standard volumetric estimate. If all $f \in \eF$ satisfy $\norm{f}_{\sX} \le LR$, then we have $e_m(\eF, d_{\sX})\le C_k LR 2^{-m / k} $ by Lemma~\ref{ent_num_est}.
Let 
\begin{align*}
  h(M):=  2^{-M / k} + \frac{1}{\sqrt{n}}\sum_{m=0}^{M}  2^{m (\frac{1}{2}-\frac{1}{k})},
\end{align*}
so that $e_M(\eF, d_{\sX}) + \sum_{m=0}^{M} \frac{2^{m / 2}}{\sqrt{n}}e_m(\eF,d_{\sX}) \le C_kLRh(M)$. For $k=1$, $\inf_M h(M) \lesssim \frac{1}{\sqrt{n}}$. For $k=2$, for $n$ sufficiently large and $M=\floor{\log_2 n}$, we have
\begin{align*}
  h(M) \lesssim \frac{1}{\sqrt{n}} +   \frac{1+\log_2 n}{\sqrt{n}}\le c \frac{\log n}{\sqrt{n}}.
\end{align*}
For $k>2$, for $n$ sufficiently large, we again let $M=\floor{\log_2 n}$ and
\begin{align*}
  h(M) &\le 2^{-M / k} +\frac{1}{\sqrt{n}} 
  2^{M(\frac{1}{2}-\frac{1}{k})}\left( \sum_{m=0}^{\infty} 2^{-m(\frac{1}{2}- \frac{1}{k})}\right) \\
  &\le C_k n^{-1/k}.
\end{align*}
The fact that $e_m(\eF, d_{\sS}) \le e_m(\eF, d_{\sX}) $ concludes the proof.

\section{Examples}\label{sec:examples}

We give two examples to illustrate the use of our chain rule. The first example is an improvement of a result from \cite{golowich2019sizeindependent}, which was an 
intermediate step in bounding Rademacher complexities of neural nets with a uniform upper bound on the spectral norm of their weight matrices. Given an $n$-tuple of points $z^n := (z_1,\ldots,z_n)$ in some space $\eZ$, the (normalized) \textit{empirical Rademacher complexity} $\hat{\eR}_n$ of a class $\eG$ of real-valued functions $g : \cZ \to \Real$
is defined as
\begin{align*}
  \hat{\eR}_n(\eG) \equiv \hat{\eR}_n(\eG; z^n) := \frac{1}{n} \lexpect{\varepsilon}{\sum_{i=1}^{n} \varepsilon_i g(z_i)},
\end{align*}
see, e.g., \cite{bartlett_rademacher_2002}. Then we have the following result, which improves on Theorem~4 in \cite{golowich2019sizeindependent}:

\begin{proposition} For some $R>0$,
 let $\eG$ be a class of functions from $\eZ$ to $[-R,R]$. Let $\eF$ be the class
 of $L$-Lipschitz functions from $[-R,R]$ to $\Real$. Then the Rademacher complexity of the composite class $\eF \circ \eG := \set{f \circ g: f\in \eF, g\in \eG}$ satisfies
 \begin{align*}
   \hat{\eR}_n(\eF \circ \eG) \lesssim L \left(\frac{R}{\sqrt{n}} + \hat{\eR}_n(\eG)\right).
 \end{align*}
\end{proposition}
\begin{proof} The result follows from \Cref{main_res_b} with $k=1$ and  $T=\set{(g(z_i))_{1\le i \le n}: g\in \eG}$. \end{proof}
\begin{remark} For comparison, the bound in \cite{golowich2019sizeindependent} has the form 
$$
\hat{\eR}_n(\eF \circ \eG) \lesssim L(\frac{R}{\sqrt{n}} + \log^{3/2}n \cdot \hat{\eR}_n(\eG)),
$$
so our improvement consists in removing the logarithmic factor in front of the second term on the right-hand side.
\end{remark}

Our second example involves functions in a \textit{Reproducing Kernel Hilbert Space} (RKHS). Due to space limitations, we can only give a brief sketch; the reader is invited to consult \cite[Chaps.~2 and 4]{cucker2007learning} for the background. Let $\mathsf{X}= \eB^{k}_{2}(\sqrt{k}R)$ for some $R > 0$. Let $(\eH_K, \inner{\cdot, \cdot}_K)$ be an RKHS associated with a Mercer kernel $K: \mathsf{X} \times \mathsf{X} \to \Real$; then we consider $\eF = I_K(\eB^{K}(\varrho))$, where $\eB^{K}(\varrho)=\set{f\in\eH_K: \norm{f}_K \le \varrho }$ is the zero-centered closed ball of radius $\varrho$ in $\eH_K$ and $I_K$ is the embedding map from $\eH_K$ into the space $\cC(\mathsf{X})$ of continuous
real-valued functions on $\mathsf{X}$ equipped with the uniform norm $\norm{\cdot}_{\sX}$. We have the following result for the Gaussian RKHS, which parallels Maurer's result in \cite[Section 3.2]{maurer_chain_2014} if we impose the same
assumption that $T$ is a projection of another RKHS function class onto a set of samples.

\begin{proposition} \label{rkhs} Consider the Gaussian kernel $K(x,y) = \exp\left(-\frac{1}{2\sigma^2}\norm{x-y}_2^2\right)$, where $\sigma^2 > 0$ is the kernel bandwidth. Then, for any $T \subset (\eB^k_2(R))^n$, 
\begin{align}\label{eq:RKHS_ber}
  \ber{\eF(T)} \lesssim \varrho \left(\frac{\ber{T}}{\sigma} + \sqrt{n}\right).
\end{align}
\end{proposition}
\begin{remark} See the Remark following \Cref{main_res_b} for the motivation behind our choices of $\sX$ and $T$.  Note, however, that we use \Cref{main_res_a} to prove the above proposition; although one can use existing bounds on the covering numbers for subsets of an RKHS (compactly embedded into a suitable space of real-valued continuous functions) in conjunction with \Cref{main_res_b}, but the constant $C_k$ appearing in \Cref{crude_est} will then have exponential dependence on $k$.
\end{remark}

\begin{proof}Let $K_x$ denote the function $y \mapsto K(x,y)$, viewed as an element of $\eH_K$. Using the reproducing kernel property, it can be shown that $\norm{K_x-K_y}_K \le \frac{1}{\sigma}\norm{x-y}$. This in turn implies that 
 $\eF$ is uniformly Lipschitz: for any $f \in \eF$ and all $x,y \in \sX$, 
\begin{align*}
  \abs{f(x)-f(y)}&=\inner{K_x-K_y,f}_K \le \norm{f}_K \norm{K_x-K_y}_K \le \frac{\norm{f}_K}{\sigma} \norm{x-y}_2.
\end{align*}
In particular, $\norm{f}_{{\rm Lip}} \le \frac{\varrho}{\sigma}$ for all $f \in \eF$.

We now apply \Cref{main_res_a} and estimate the terms dependent on $\eH_K$. Since $\norm{S}_{\infty,\infty} \le \norm{T}_{\infty,\infty}$, we have $S \subset (\eB^k_2(\sqrt{k}R))^n$. For any $s,t \in S$, we can use the reproducing kernel property, the independence of $\varepsilon_i$, and the Lipschitz continuity of $t \mapsto K_t$ to show that
\begin{align}
  \expect{\sup_{f \in \eF} \sum_{i=1}^{n} \varepsilon_i (f(s_i)-f(t_i))}
 & =\expect{\sup_{f \in \eF} \sum_{i=1}^{n} (\varepsilon_i \inner{K_{s_i}-K_{t_i},f}_K )}\\
 & = \expect{\sup_{f \in \eF} \inner{f, \sum_i \varepsilon_i(K_{s_i}-K_{t_i})}_K } \\
 & =\varrho\expect{\norm{\sum_i \varepsilon_i (K_{s_i}-K_{t_i})}_K}\label{risze}\\
 &\le \varrho\sqrt{\expect{ \norm{ \sum_i \varepsilon_i (K_{s_i}-K_{t_i}) }_K^2 }}\label{jensen}\\
 &\le \varrho \sqrt{\sum_i \norm{K_{s_i}-K_{t_i}}_K^2}\\
 & \le \frac{\varrho}{\sigma} \norm{s-t}_2
\end{align}
for all $f \in \eF$ and all $s,t \in S$, where (\ref{risze}) uses the Riesz representation theorem, (\ref{jensen}) uses Jensen's inequality, and the last two inequalities use independence of $\varepsilon_i$ and Lipschitz continuity of $x \mapsto K_x$, which implies that $\comr{\eF,S} \le \frac{\varrho}{\sigma}$. Moreover, since $\eF = I_K(\cB^K(\varrho))$, we have $\ber{\eF(\tau)} \le \sqrt{n}\varrho$ for any $\tau\in S$, see, e.g., \cite{bartlett_rademacher_2002}. Therefore, applying \Cref{main_res_a} and using the fact that $\Delta_2(S) \le \sqrt{2\pi}\gau{S} \lesssim \ber{T}$, we obtain \eqref{eq:RKHS_ber}. \end{proof}

\appendix
\section{Auxiliary results and proofs}

\label{app:aux}

\subsection{Proof of Lemma~\ref{gen_chain_appr}}
\begin{proposition}\label{union_comp}
  For some $u, \varrho, \zeta >0, w \in \Integer_+ $,
  let 
  \begin{align}
    p(u) := \sum_{m\ge 1} 2^{2^{m+1+w}} \e^{-u^2 2 ^{m-1}}\mbox{ and } q(u) := p(u)\wedge 1.
  \end{align}
  If a nonnegative random variable $Y$ satisfies $\prob{Y > u \varrho + \zeta} \le q(u)$,
  then we have 
  \begin{align}
    \mathbf{E}Y \le C_w \varrho + \zeta.
  \end{align}
\end{proposition}
\begin{proof}
  
Observe that
for $m \ge 1$ and for $u \ge  C_w$ such that $C_w^2\ge 2^{3+w}$, we have
\begin{align*}
  u^{2} 2^{m-1} \geq \frac{u^{2}}{2}+u^{2} 2^{m-2} \geq \frac{u^{2}}{2}+2^{m+1+w},
\end{align*}
which yields
\begin{align*}
  q(u) = (p(u)\idf{u\ge C_w} + p(u )\idf{u < C_w} ) \wedge 1 \le c \idf{u\ge C_w} \e^{-u^2 / 2} + \idf{u < C_w}.
\end{align*}
Now we convert the tail bound to a moment bound. Since $Y$ is nonnegative, we have
\begin{align*}
  \mathbf{E}Y = \int_{0}^{\infty} \prob{Y > y} \diff y = \int_{- \zeta / \varrho}^{\infty} \prob{Y > u \varrho + \zeta}\, \varrho\, \diff u \le \int_{- \zeta / \varrho }^{\infty} q(u)\, \varrho\, \diff u \le C_w \varrho + \zeta,
\end{align*}
which concludes the proof.
\end{proof}

\begin{proof}[\textit{Proof of Lemma~\ref{gen_chain_appr}}]
  Assume for some $(\cA_m)$,
  $\sup_{t \in T} \sum_{m=0}^{\infty} 2^{m / 2} \Delta_d(A_m(t)) < \infty$. 
  It suffices to consider finite $T$ since the result can be extended to countable case by standard
  convergence arguments (see \eg \cite[Chapter 2]{talagrand_upper_2014}).
  Fix $\tau\in T$.
For any $m\ge 0$, 
 for each $A(t) \in \cA_m$ uniquely containing $t$, we arbitrarily 
take a point $\pi_m(t):= \pi_m(A(t)) \in A(t)$
as the approximation point for $t$ or $A(t)$. For sufficiently large $M\in \Natural$,
we have $\pi_M(t)= t$. Then for any $t \in T$, 
by choosing $\tau= \pi_0(t)$, we have the usual chaining decomposition $Y_t - Y_{\tau} = \sum_{m=1}^{M} Y_{\pi_m(t)} - Y_{\pi_{m-1}(t)}$. Let
\begin{align*}
  \varrho:=\sup_{t \in T} \sum_{m=1}^{M}\varrho_{m}(t):= 
  \sup_{t \in T} \sum_{m=1}^{M} 2^{m/2} d(\pi_m(t), \pi_{m-1}(t)),\\
  \zeta :=\sup_{t \in T} \sum_{m=1}^{M} \zeta_m(t):=
  \sup_{t \in T} \sum_{m=1}^{M} \sqrt{2 \log \kappa}d(\pi_m(t), \pi_{m-1}(t)) .
\end{align*}
For any $u> 0$, let 
\begin{align*}
  \Omega_u:=\set{\sup_{t \in T} Y_t - Y_{\tau} > \sup_{t \in T} \sum_{m=1}^{M} u\varrho_m(t) + \zeta_m(t)}.
\end{align*}
Since $u\varrho + \zeta \ge \sup_{t \in T} \sum_{m=1}^{M} u\varrho_m(t) + \zeta_m(t)$,
we obtain
\begin{align*}
  \prob{\sup_{t \in T} Y_t - Y_{\tau} > u \varrho + \zeta}
  \le \prob{\Omega_u}.
\end{align*}
Note that the occurrence of $\Omega_u$ implies that 
\begin{align*}
  \exists\, t\in T,\, \exists \,m\ge 1, \,   Y_{\pi_m(t)}-Y_{\pi_{m-1}(t)}
  >  u\varrho_m(t) + \zeta_m(t).
\end{align*}
Also, the size of $\set{\pi_m(t), \pi_{m-1}(t): t \in T}$ can be bounded by 
$2^{2^{m}}\cdot 2^{2^{m}}= 2^{2^{m+1}}$.
Using the union bound and the increment condition of $Y_t$ in (\ref{subgau_cond}), we have
\begin{align*}
  \prob{\Omega_u} 
  &\le \sum_{m=1}^{M} 2^{2^{m+1}} 
  \prob{Y_{\pi_m(t)}-Y_{\pi_{m-1}(t)}
  >  u\varrho_m(t) + \zeta_m(t)}\\ 
  &\le \sum_{m=1}^{M} 2^{2^{m+1}}
  \kappa \exp \left(- \frac{(u 2^{m / 2} + \sqrt{2 \log \kappa})^2}{2}\right) \\
  &\le \sum_{m=1}^{M} 2^{2^{m+1}} \e^{-u^2 2^{m-1}}=: p(u).
\end{align*}
Hence, $\prob{\Omega_u} \le p(u) \wedge 1 =: q(u)$.
Since $\sup_t Y_t - Y_{\tau}$ is nonnegative, we use Proposition~\ref{union_comp} to obtain
\begin{align*}
  \expect{\sup_{t \in T} Y_t - Y_{\tau}} \le c\varrho + \zeta \le c \sup_{t \in T} \sum_{m=0}^{\infty} 2^{m / 2} \Delta_d(A_m(t)) + 
  c \sup_{t \in T} \Delta_d(T)\sqrt{\log \kappa},
\end{align*}
where, in the last inequality, the second term is due to 
the fact that, since for any $t$ the series $\sum_{m} 2^{m / 2} \Delta_d(A_m(t))$ converges, so does 
$\sum_{m} \Delta_2(A_m(t))$.
Then $\sum_{m} \Delta_d(A_m(t)) = \Delta_d (T)\sum_{m} (\Delta_d (A_m(t)) / \Delta_d(T))$
is dominated by the
first term of the series $\Delta_d(T_2)$ up to some constant.
\end{proof}

\bibliographystyle{plain}
\bibliography{chain_rule.bbl}

\end{document}

%% file: despkg.tex
\usepackage{amsmath,amsthm,amssymb}
\usepackage{enumerate}
\usepackage[mathscr]{euscript}
\usepackage{tikz}
\usepackage{xcolor}
\usepackage[english]{babel}
\usepackage{cleveref}